\documentclass[a4paper,12pt,onecolumn]{article}

\usepackage{hyperref}
\usepackage[vmargin=2cm,hmargin=2cm,headheight=14.5pt,top=2cm,headsep=.5cm]{geometry}
\usepackage{mathptmx}
\usepackage{bm}
\usepackage{empheq}
\usepackage{stackrel}
\usepackage{cases}
\usepackage{mathtools}
\usepackage{amsthm,amsmath,amscd}
\usepackage{makeidx}
\usepackage[charter]{mathdesign}
\usepackage{perpage}
\usepackage{tikz-cd}
\tikzcdset{every label/.append style = {font = \small}}
\usepackage[symbol]{footmisc}

\MakePerPage[2]{footnote}
\usepackage{graphicx}
\DeclareMathSizes{12}{12}{8}{6}

\usepackage{cite}
\usepackage{url}

\newtheoremstyle{ptheorem}{1em}{0em}{\itshape}{}{\bfseries}{.}{.5em}{\thmname{#1}\thmnumber{ #2}\thmnote{ (\hspace{-.01pt}{#3})}}

\theoremstyle{ptheorem}

\newtheorem{thm}{thm}[section]

\newtheorem{cor}[thm]{cor}

\newtheoremstyle{hdef}{1em}{0em}{}{}{\bfseries}{.}{.5em}{\thmname{#1}\thmnumber{ #2}\thmnote{ (\hspace{-.01pt}{#3})}}
\theoremstyle{hdef}

\newtheorem{dfn}[thm]{dfn}
\newtheorem{rem}[thm]{rem}

\makeatletter
\newtheoremstyle{premark}{1em}{0em}{
\addtolength{\@totalleftmargin}{1.5em}
\addtolength{\linewidth}{-1.5em}
\parshape 1 1.5em \linewidth}{}{\scshape}{.}{.5em}{}
\makeatother

\theoremstyle{premark}

\newtheorem{exa}[thm]{exa}

\numberwithin{equation}{section}
\numberwithin{figure}{section}

\DeclareMathOperator{\Id}{Id}

\DeclareMathOperator{\dif}{d}

\DeclareMathOperator{\PD}{PD}

\newcommand{\cA}{{\mathcal A}}

\newcommand{\cC}{{\mathcal C}}

\newcommand{\cF}{{\mathcal F}}

\newcommand{\cM}{{\mathcal M}}

\newcommand{\bC}{{\mathbb C}}

\newcommand{\bF}{{\mathbb F}}

\newcommand{\bN}{{\mathbb N}}

\newcommand{\bR}{{\mathbb R}}

\renewcommand{\a}{\alpha}
\renewcommand{\b}{\beta}

\renewcommand{\l}{\lambda}

\newcommand{\ol}{\overline}

\renewcommand{\(}{\left(}
\renewcommand{\)}{\right)}

\newcommand{\til}{\tilde}
\newcommand{\Lsp}[1]{\operatorname{L^{#1}}}
\begin{document}
\title{Green's Functions of Partial Differential Equations with Involutions}

\author{
F. Adri\'an F. Tojo\footnote{Partially supported by Xunta de Galicia (Spain), project EM2014/032.} \\
\normalsize
Departamento de An\'alise Ma\-te\-m\'a\-ti\-ca, Facultade de Matem\'aticas,\\ 
\normalsize Universidade de Santiago de Com\-pos\-te\-la, Spain.\\ 
\normalsize e-mail: fernandoadrian.fernandez@usc.es\\
Pedro J. Torres\footnote{Partially supported by Spanish MICINN Grant with FEDER funds MTM2014-52232-P.}\\
\normalsize
Departamento de Ma\-te\-m\'a\-ti\-ca Aplicada, Facultad de Ciencias,\\ 
\normalsize Universidade de Granada, Spain.\\ 
\normalsize e-mail: ptorres@ugr.es\\
}
\date{}

\maketitle

\begin{abstract}
In this paper we develop a way of obtaining Green's functions for partial differential equations with linear involutions by reducing the equation to a higher-order PDE without involutions. The developed theory is applied to a model of heat transfer in a conducting plate which is bent in half.
\end{abstract}


\section{Introduction}
The study of differential equations with involutions dates back to the work of Silberstein \cite{Sil} who, in 1940, obtained the solution of the equation $f(x)=f(1/x)$. In the field of differential equations there has been quite a number of publications (see for instance the monograph on the subject of reducible differential equations of Wiener \cite{Wie2}) but most of them relate to ordinary differential equations (ODEs). There has also been some work in partial differential equations (PDEs), for instance \cite{Wie2} or \cite{Bur}, where they study a PDE with reflection.\par
In what Green's functions for equations with involutions is concerned, we find in \cite{Cab4} the first Green's function for ODEs with reflection and in \cite{CTMal} we have a framework that allows the reduction of any differential equation with reflection and constant coefficients. This setting is established in a general way, so it can be used as well for other operators (the Hilbert transform, for instance) or in other yet unexplored problems, like PDEs \cite{Kir}. In this work we take this last approach and find a way of reducing general linear PDEs with linear involutions to usual PDEs.\par
The paper is structured as follows. In Section 2 we develop an abstract framework, with definitions and adequate notation in order to treat linear PDEs as elements of a vector space consisting of symmetric tensors. This will allow us to systematize the algebraic transformations necessary in order to obtain the desired reduction of the problem. In Section 3 we start providing a simple example that shows how the general process works and then prove the main result of the paper, Theorem \ref{theoremO2}, that permits a general reduction in the case of order two involutions. We end the Section with a problem with an order $3$ involution (Example \ref{o3ex}), illustrating that the same principles could be applied to higher order involutions. Finally, in Section 4, we describe a way to obtain Green's functions for PDEs with linear involutions and apply it to a model of the process of heat transfer in a conducting plate which is bent in half with the two halves separated by some insulating material. We study the problem for different kinds of boundary conditions and a general heat source.
\section{Definitions and notation}
\subsection{Derivatives}
Let $\bF$ be $\bR$ or $\bC$, $n\in\bN$ and $\Omega\subset V:=\bF^n$ a connected open subset. For $p\ge 2$, note by $V^{\odot p}$ the space of symmetric tensors or order $p$, that is, the space of tensors of order $p$ modulus the permutations of their components.  We note $V^{\odot 1}=V$ and $V^{\odot 0}=\bF$. For the convenience of the reader, we summarize now the properties and operations of the symmetric tensors:
\begin{itemize}
\item $V^{\odot p}:=\{v_1^1\odot \cdots\odot v^k_1+\dots +v_r^1\odot \cdots\odot v^k_r\ :\ v_j^s\in\bF^n;\ j=1,\dots,r;\ s=1,\dots,k;\ r,k\in\bN\}$.
\item $(v_1^1\odot \cdots\odot v^k_1)\odot(v_1^{k+1}\odot \cdots\odot v^p_1)=v_1^1\odot \cdots\odot v^p_1;\ v_1^s\in\bF^n;\ s=1,\dots,p;\ p\in\bN$.
\item $v_1\odot v_2=v_2\odot v_1;\ v_1,v_2\in\bF^n$.
\item $\l(v_1\odot v_2)=(\l v_1)\odot v_2;\ v_1,v_2\in\bF^n$.
\item $(v_1+v_2)\odot v_3=v_1\odot v_3+v_2\odot v_3;\ v_1,v_2,v_2\in\bF^n$.
\item $0\odot v_1=0;\ v_1\in\bF^n$.
\end{itemize}
With these properties, $V^{\odot p}$ is an $\bF$-vector space of dimension ${n + p - 1 \choose p}$.\par

 For every $v=(v_1,\dots,v_n)\in V$, we define the directional derivative operator as
\begin{center}\begin{tikzcd}[row sep=tiny]
\cC^1(\Omega,\bF) \arrow{r}{D_v} & \cC(\Omega,\bF)\\
y \arrow[mapsto]{r} & v_1\frac{\partial y}{\partial x^1}+\cdots+v_n\frac{\partial y}{\partial x^n}
\end{tikzcd}
\end{center}
If $\nabla y$ denotes the gradient vector of $y$, then $D_v(y)=v^T\nabla y$. Observe that $D_{\l u+v}=\l D_u+D_v$ for every $u,v\in\bF^n$ and $\l\in\bF$, that is, $D_v$ is linear in $v$. Also, for $u,v\in\bF^n$, if $y\in\cC^2(\Omega,\bF)$, then $D_u(D_v y)=D_v(D_u y)$. Furthermore, $D_u\circ D_v$ is bilinear --that is, linear in both $u$ and $v$, so we can write the identification $D_v\circ D_u\equiv D^2_{v\odot u}$, where $v\odot u$ denotes de symmetric tensor product of $u$ and $v$. In the same way, we define the composition of higher order derivatives by $D^p_{\omega_2}\circ D^q_{\omega_1}=D^{p+q}_{\omega_2\odot \omega_1}$ where $\omega_1\in V^{\odot q}$ and $\omega_2\in V^{\odot p}$, $p,q\in\bN$.\par

In this way, a linear partial differential equation is given by
\begin{equation}Ly:=\sum_{k=0}^mD^k_{\omega_k}y=0,\label{lineq}\end{equation}
where $\omega_k\in V^{\odot k}$ for $k=1,\dots,m$ and $D^0_{\omega_0}u\equiv\omega_0 u$ where $\omega_0\in\bF$ (that is, $V^{\odot 0}:=\bF$). Now, the operator $L$ can be identified with $\omega_0+ \omega_1+\cdots+ \omega_n$, which is an element of the  symmetric tensor algebra
\[ S^*V:=\bigoplus_{k=0}^{\infty}V^{\odot n} =\mathbb{F} \oplus V \oplus \left(V\odot V\right) \oplus \left(V\odot V\odot V\right) \oplus \cdots\]
It is interesting to point out the the Hilbert space completion of $S^*V$, that is, $F_+(V):=\ol{S^*V}$, is called the \emph{symmetric} or \emph{bosonic Fock space}, which is widely used in quantum mechanics \cite{Fock}.
\subsection{Involutions}
\begin{dfn} Let $\Omega$ be a set and $A:\Omega\to \Omega$, $p\in\bN$, $p\ge2$. We say that $A$  is an \emph{order} $p$ \emph{involution}  if
\begin{enumerate}
\item $A^p\equiv A\circ\stackrel{\stackrel{p}{\smile}}{\cdots}\circ A=\Id$,
\item $A^j\ne\Id,\  j=1,\dots,p-1$.
\end{enumerate}
\end{dfn}
We will consider linear involutions in $\bF^n$. They are characterized by the following theorem.
\begin{thm}[\protect{\cite{Amir}}]\label{theoremc}A necessary and sufficient condition for a linear transformation $A$ on a finite dimensional complex vector space $V$ to be an involution of order $p$ is that $A=\a_1P_1+\cdots+\a_kP_k$ where $\a_j$ is a $p$-th root of the unity, and $P_1,\dots,P_k$ are projections such that $P_jP_l=0$, $i\ne j$ and $P_1+\cdots + P_k=\Id$.
\end{thm}
\begin{rem} As an straightforward consequence of this result we have that there are only order two linear involutions in $\bR^n$. This is because the only real $p$-th roots of the unity are contained in $\{\pm 1\}$.
\end{rem}
The characterization provided in Theorem  \ref{theoremc} can be rewritten in the following way.
\begin{cor}A necessary and sufficient condition for a linear transformation $A$ on $V$ to be an involution of order $p$ is that $A=U^{-1}\Lambda U$ where $\Lambda,U\in\cM_n(\bF)$, $U$ is invertible and $\Lambda$ is a diagonal matrix where the elements of the diagonal are $p$-th roots of the unity.
\end{cor}
\begin{proof} Consider the characterization of involutions given by Theorem \ref{theoremc}. Take the vector subspaces $H_j:=P_jV$, $j=1,\dots,k$. Then, $V=H_1\oplus\dots\oplus H_k$. Take $U^{-1}$ to be the matrix of which its columns are, consecutively, a basis of $H_k$. Hence, $A=U^{-1}\Lambda U$ where $\Lambda$ is a diagonal matrix of diagonal \[(\a_1,\dots,\a_1,\a_2,\dots,\a_2,\dots,\a_k,\dots,\a_k),\] where every $\a_j$ is repeated accorollaryding to the dimension of $H_k$.
\end{proof}
\subsection{Pullbacks and equations}
Let $\cF(\Omega,\bF)$ be the set of functions from $\Omega\subset\bF^n$ to $\bF$. We define the pullback operator by a function $\varphi\in\cF(\Omega,\Omega)$ as
\begin{center}\begin{tikzcd}[row sep=tiny]
\cF^1(\Omega,\bF) \arrow{r}{\varphi^*} & \cF(\Omega,\bF)\\
y \arrow[mapsto]{r} &y\circ\varphi
\end{tikzcd}
\end{center}
Assume $A$ is a linear  order $p$ involution on $\Omega$ ($\Omega$ has to be such that $\Omega=A(\Omega)$). From now on, we will omit the composition signs. Observe that, for $v\in V$, $x\in\Omega$ and $y\in\cC^1(\Omega,\bF)$,
\begin{align*}((D_vA^*)y)(x) & =D_v(y(Ax))=v^T\nabla( y(Ax))=v^TA^T\nabla y(Ax) \\ & =(Av)^T\nabla y(Ax)=D_{Av}\nabla y(Ax)=(A^*D_{Av})y(x),\end{align*}
or, written briefly, $D_vA^*=A^*D_{Av}$. All the same, for $v_1,\dots,v_j\in V$,
\[D^j_{v_1\odot\dots\odot v_j}A^*=A^*D^j_{Av_1\odot\dots\odot Av_j}.\]
If $\omega_k=v_1\odot\dots\odot v_k\in V^{\odot k}$, we denote $A\omega_k\equiv Av_1\odot\dots\odot Av_j$. This way, $D^j_{\omega_k}A^*=A^*D^j_{A\omega_k}$.\par
We can consider now linear partial differential equations with linear involutions of the form
\[Ly:=\sum_{j=0}^{p-1}\sum_{k=0}^m(A^*)^jD^k_{\omega^j_k}y=0,\]
where $\omega^j_k\in V^{\odot k}$ for $k=0,\dots,m$; $j=0,\dots,p-1$. This time we can identify $L$ with
\[\left(\omega_1^0+\cdots+\omega_m^0,\,\omega_1^1+\cdots+\omega_m^1,\,\dots,\,\omega_1^{p-1}+\cdots+\omega_m^{p-1}\right)\in(S^*V)^p.\]
The interest in these equations appears when they can be reduced to usual partial differential equations.
\begin{dfn}[\cite{CTMal}]
If $\bF[D]$ is the ring of polynomials on the usual differential operator $D$ and $\cA$ is any operator algebra containing $\bF[D]$, then an equation $L x=0$, where $L\in\cA$, is said to be a \emph{reducible differential equation}\index{reducible differential equation} if there exits $R\in\cA$ such that $RL\in\bF[D]$.
\end{dfn}

In our present case, the first projection of the algebra $(S^*V)^p$ is precisely the algebra of partial differential operators on $n$ variables $\PD_n[\bF]$, so we want to find elements $R\in (S^*V)^p$ such that they nullify the last $p-1$ components of $L$.
\section{Reducing the operators}
We start with an illustrative example.
\begin{exa}\label{simplex}
Let $V=\bR^2$, $v=(v_1,v_2)\in V$ and
\[A=\begin{pmatrix}1 & 0  \\ 0 & -1\end{pmatrix}.\]
$A$ is an order $2$ involution. Consider the equation
\begin{equation}\label{1ex1eq}v_1\frac{\partial y}{\partial x_1}(x)+v_2\frac{\partial y}{\partial x_2}(x)+ y(Ax)=0,\ x=(x_1,x_2)\in\bR^2.\end{equation}
Here we work with the operator $L=D_v+A^*$. Take then $R=D_{-Av}+A^*$ and consider the identity operator $\Id$. We have that
\begin{align*} RL = & (D_{-Av}+A^*)(D_v+A^*)=D_{-Av}D_v+A^*D_v+D_{-Av}A^*+(A^*)^2 \\ = & D_{-Av\odot v}+A^*D_v+A^*D_{-AAv}+\Id =  D_{-Av\odot v}+A^*D_{v}+A^*D_{-v}+\Id \\ = & D_{-Av\odot v}+A^*D_v-A^*D_{v}+\Id = D_{-Av\odot v}+\Id.\end{align*}
Hence, every two-times differentiable solution of equation \eqref{1ex1eq} has to be a solution of the partial differential equation
\begin{equation*}-v_1^2\frac{\partial^2 y}{\partial x_1^2}(x)+v_2^2\frac{\partial^2 y}{\partial x_2^2}(x)+y=0,\ x=(x_1,x_2)\in\bR^2.\end{equation*}
\end{exa}
\begin{rem} With the notation we have introduced, it is extremely important the use of parentheses. Observe that every $\omega\in (\bF^n)^{\odot k}$ can be expressed as $\omega=v_1^1\odot \cdots\odot v^k_1+\dots +v_r^1\odot \cdots\odot v^k_r$ for some $v_j^s\in\bF^n$, $j=1,\dots,r$, $s=1,\dots,k$; $r,k\in\bN$. Hence, for $c\in\bF$,
\begin{align*}(cA)\omega= & cAv_1^1\odot \cdots\odot cA v^k_1+\dots +cAv_r^1\odot \cdots\odot cAv^k_r \\= &c^k(Av_1^1\odot \cdots\odot Av^k_1)+\dots +c^k(Av_r^1\odot \cdots\odot Av^k_r)=c^k(A\omega)\equiv c^kA\omega.\end{align*}
\end{rem}
\begin{thm}\label{theoremO2}
	Let $A$ be an order $2$ linear involution on $\bF^n$. Let $L\in(S^*V)^p$ be defined as in \eqref{lineq}. Then there exists $R\in(S^*V)^p$ defined as
	\[Ry:=\sum_{j=0}^{p-1}\sum_{k=0}^m(A^*)^jD^k_{\xi^j_k}y=0,\]
	where $\xi^0_k=-A\omega^0_k$, $\xi^1_k=\omega^1_k$, for $k=0,1,\dots$, such that $RL\in \PD_n[\bF]$. Furthermore, $L$ and $R$ commute.
\end{thm}
\begin{proof} For convenience, define $\xi^j_k$ and $\omega^j_k$ outside the index range $j=0,\dots,p-1$, $k=0,\dots m$ to be zero. In general,
	\begin{align*}
		RL= & \sum_{l=0}^{p-1}\sum_{r=0}^m(A^*)^lD^r_{\xi^l_r}\left(\sum_{j=0}^{p-1}\sum_{k=0}^m(A^*)^jD^k_{\omega^j_k}\right)=\sum_{l,j=0}^{p-1}\sum_{r,k=0}^m(A^*)^lD^r_{\xi^l_r}(A^*)^jD^k_{\omega^j_k} \\ = & \sum_{l,j=0}^{p-1}\sum_{r,k=0}^m(A^*)^{l+j}D^r_{A^j\xi^l_r}D^k_{\omega^j_k}
		= \sum_{l,j=0}^{p-1}\sum_{r,k=0}^m(A^*)^{l+j}D^{r+k}_{A^j\xi^l_r\odot\omega^j_k}\\= &  \sum_{l,j=0}^{p-1}(A^*)^{l+j}\(\sum_{s=0}^{2m}\sum_{k=0}^sD^{s}_{A^j\xi^l_{s-k}\odot\omega^j_k}\)=\sum_{l,j=0}^{p-1}(A^*)^{l+j}\(\sum_{s=0}^{2m}D^{s}_{\sum_{k=0}^sA^j\xi^l_{s-k}\odot\omega^j_k}\).
	\end{align*}
In the particular case $p=2$, we have that
\begin{align*}
	RL= & \sum_{s=0}^{2m}D^{s}_{\sum_{k=0}^s\xi^0_{s-k}\odot\omega^0_k}+ \sum_{s=0}^{2m}D^{s}_{\sum_{k=0}^sA\xi^1_{s-k}\odot\omega^1_k} \\ & + A^*\(\sum_{s=0}^{2m}D^{s}_{\sum_{k=0}^s\xi^1_{s-k}\odot\omega^0_k}+ \sum_{s=0}^{2m}D^{s}_{\sum_{k=0}^sA\xi^0_{s-k}\odot\omega^1_k}\)\\
	= & \sum_{s=0}^{2m}D^{s}_{\sum_{k=0}^s\(\xi^0_{s-k}\odot\omega^0_k+A\xi^1_{s-k}\odot\omega^1_k\)}+ A^*\(\sum_{s=0}^{2m}D^{s}_{\sum_{k=0}^s\(\xi^1_{s-k}\odot\omega^0_k+A\xi^0_{s-k}\odot\omega^1_k\)}\).
\end{align*}
So it is enough to check that, for $s=0,\dots,2m$,
\[\sum_{k=0}^s\(\xi^1_{s-k}\odot\omega^0_k+A\xi^0_{s-k}\odot\omega^1_k\)=0.\]
Substituting the $\xi_k^j$ by their given values,
\begin{align*}& \sum_{k=0}^s\(\xi^1_{s-k}\odot\omega^0_k+A\xi^0_{s-k}\odot\omega^1_k\) =   \sum_{k=0}^s\(\omega^1_{s-k}\odot\omega^0_k-A^2\omega^0_{s-k}\odot\omega^1_k\) \\ = & \sum_{k=0}^s\(\omega^1_{s-k}\odot\omega^0_k-\omega^0_{s-k}\odot\omega^1_k\) = \sum_{k=0}^s\omega^1_{s-k}\odot\omega^0_k-\sum_{k=0}^s\omega^0_{s-k}\odot\omega^1_k \\ = & \sum_{k=0}^s\omega^1_{s-k}\odot\omega^0_k-\sum_{k=0}^s\omega^0_{k}\odot\omega^1_{s-k}=0.
\end{align*}
\par
Let us see that $L$ and $R$ commute.
\[
LR=  \sum_{s=0}^{2m}D^{s}_{\sum_{k=0}^s\(\omega^0_{s-k}\odot\xi^0_k+A\omega^1_{s-k}\odot\xi^1_k\)}+ A^*\(\sum_{s=0}^{2m}D^{s}_{\sum_{k=0}^s\(\omega^1_{s-k}\odot\xi^0_k+A\omega^0_{s-k}\odot\xi^1_k\)}\).\]
Now,
\begin{align*}  & \sum_{k=0}^s\(\omega^0_{s-k}\odot\xi^0_k+A\omega^1_{s-k}\odot\xi^1_k\)=\sum_{k=0}^s\omega^0_{k}\odot\xi^0_{s-k}+\sum_{k=0}^sA\omega^1_{s-k}\odot\omega^1_k \\ = &  \sum_{k=0}^s\xi^0_{s-k}\odot\omega^0_{k}+\sum_{k=0}^sA\xi^1_{s-k}\odot\omega^1_k.\end{align*}
On the other hand,
\begin{align*} &  \sum_{k=0}^s\(\omega^1_{s-k}\odot\xi^0_k+A\omega^0_{s-k}\odot\xi^1_k\)=\sum_{k=0}^s\(\omega^1_{s-k}\odot(-A\omega^0_k)+A\omega^0_{s-k}\odot\omega^1_k\)  \\= &\sum_{k=0}^s\(-\omega^1_{s-k}\odot A\omega^0_k+A\omega^0_{s-k}\odot\omega^1_k\)=\sum_{k=0}^s\(-\omega^1_{k}\odot A\omega^0_{s-k}+A\omega^0_{s-k}\odot\omega^1_k\)=0.\end{align*}
Hence, the result is proven.
\end{proof}
Similar reductions can be found for higher order involutions, although the coefficients may have a much more complex expression.
\begin{exa}\label{o3ex} Let $A$ be and order $3$ linear involution in $\bC^n$, $v\in\bC^n\backslash\{0\}$ and consider the operator $L=D_v+A^*$. Define now 
	\[R:=D_{v\odot A^2v}-A^*D_{A^2v}+(A^*)^2.\]
Observe that second derivatives occur in $R$ but not in $L$. We have that
\begin{align*}RL= & D_{v\odot A^2v}D_v-A^*D_{A^2v}D_v+(A^*)^2D_v+D_{v\odot A^2v}A^*-A^*D_{A^2v}A^*+(A^*)^2A^*\\ = & D_{v\odot v\odot A^2v}-A^*D_{v\odot A^2v}+(A^*)^2D_v+A^*D_{v\odot A^2v}-(A^*)^2D_{v}+\Id \\ = & D_{v\odot v\odot A^2v}+\Id.\end{align*}
Unfortunately, we do not have commutativity in general:
\begin{align*}LR= & D_vD_{v\odot A^2v}-D_vA^*D_{A^2v}+D_v(A^*)^2+A^*D_{v\odot A^2v}-(A^*)^2D_{A^2v}+\Id\\ = &
D_{v\odot v\odot A^2v}-A^*D_{A^*v\odot A^2v}+(A^*)^2D_{A^2v}+A^*D_{v\odot A^2v}-(A^*)^2D_{A^2v}+\Id \\ = & D_{v\odot v\odot A^2v}+A^*D_{(v-A^*v)\odot A^2v}+\Id .\end{align*}
 In the particular case $v$ is a fixed point of $A$, $RL=LR$.\par The obtaining of a general expression for associated operators in the case of order 3 involutions and the conditions under which such operators commute is an interesting open problem.
\end{exa}
\section{Green's functions}
Consider now the following problem
\begin{equation}\label{prineq}
Lu=h;\  B_\l u=0,\ \l\in\Lambda,
\end{equation}
where $L\in(S^*V)^p$, $h\in \Lsp{1}(\bF^n,\bF)$, the $B_\l:\cC(\bF^n,\bF)\to\bF$ are linear functionals, $\l\in\Lambda$ and $\Lambda$ is an arbitrary set.

Let $R\in(S^*V)^p$, $f\in \Lsp{1}(\bF^n,\bF)$ and consider the problem
\begin{equation}\label{prineq2}
RLv=f;\  B_\l v=0,\  B_\l Rv=0,\ \l\in\Lambda.
\end{equation}
Given a function $G:\bF^n\times\bF^n\to\bF$, we define the operator $H_G$ such that $H_G(h)|_x:=\int_{\bF^n} G(x,s)h(s)\dif s$ for every $h\in\Lsp{1}(\bF^n,\bF)$, assuming such an integral is well defined. Also, given an operator $R$ for functions of one variable, define the operator $R_\vdash$ as $R_\vdash G(t,s):=R(G(\cdot,s))|_{t}$ for every $s$, that is, the operator acts on $G$ as a function of its first variable. \par
We have now the following theorem relating problems \eqref{prineq} and \eqref{prineq2}. The proof for the case of ordinary differential equations can be found in \cite{CTMal}. The case of PDEs is analogous.
\begin{thm}\label{theoremmostgen}
	Let $L,\,R\in(S^*V)^p$, $h\in\Lsp{1}(\bF^n,\bF)$. Assume $L$ commutes with $R$ and that there exists $G$ such that $H_G$ is well defined satisfying\par
	$\begin{array}{rl}
	(I) & (RL)_\vdash G=0,\\
	(II) & B_{\l\,\vdash} G=0,\ \l\in\Lambda,\\
	(III) & ( B_\l R)_\vdash G=0,\ \l\in\Lambda,\\
	(IV) & RLH_Gh=H_{(RL)_\vdash G}h+h,\\
	(V) & LH_{R_\vdash G}h=H_{L_\vdash R_\vdash G}h+h,\\
	(VI) &  B_\l H_G= H_{B_{\l\,\vdash} G},\ \l\in\Lambda,\\
	(VII) &  B_\l RH_G= B_\l H_{R_\vdash G}=H_{( B_\l R)_\vdash G},\ \l\in\Lambda.
	\end{array}$\par
	Then, $v:=H_Gf$ is a solution of problem \eqref{prineq2} and $u:=H_{R_\vdash G}h$ is a solution of problem \eqref{prineq}.
\end{thm}
\subsection{A model of stationary heat transfer in a bent plate}
We now consider a circular plate which is bent in half, with each of the two distinct halves separated by a very small distance which may be filled with some kind of (imperfect) heat insulating material (see Figure \ref{figplate}).\par

\begin{figure}\begin{center}\includegraphics[width=.5\textwidth]{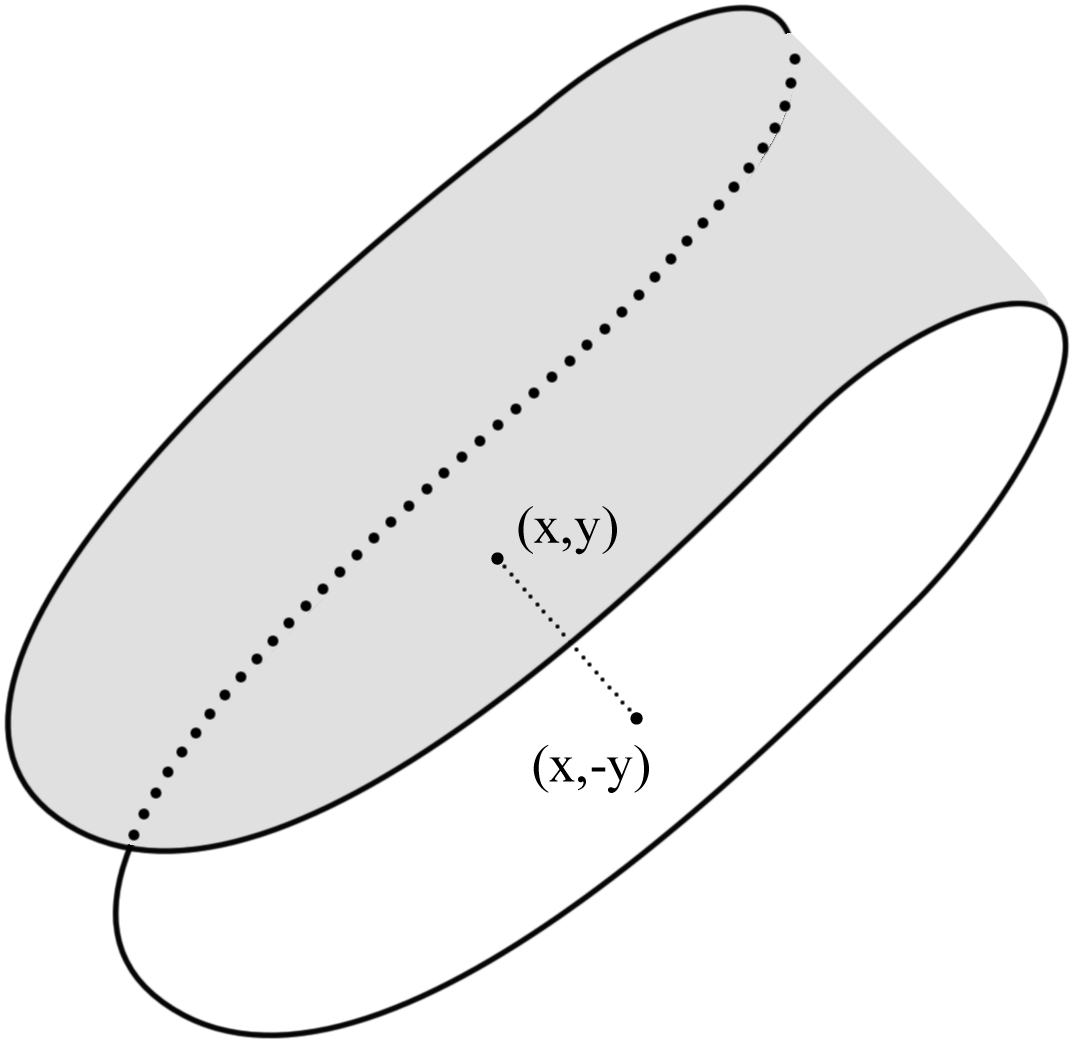}\end{center}
	\label{figplate}\caption{A section of the plate bent in half.}\end{figure}
The heat equation which determines the temperature $u$ on the plate for this situation is given by
\[\frac{\partial{u}}{\partial t}(t,x,y)=\a\left[\frac{\partial^2 u}{\partial x^2}(t,x,y)+\frac{\partial^2 u}{\partial y^2}(t,x,y)\right]+\b[u(t,x,-y)-u(t,x,y)],\]
where 
\[\frac{\partial{u}}{\partial t}(t,x,y)=\a\left[\frac{\partial^2 u}{\partial x^2}(t,x,y)+\frac{\partial^2 u}{\partial y^2}(t,x,y)\right],\]
is the usual heat equation with heat transfer coefficient $\a>0$ and the term that goes with $\b>0$ relates to the heat transfer from the corollaryresponding point in the other half of the plate due to Newton's law of cooling.\par
If we consider the associated stationary problem
\[\a\left[\frac{\partial^2 u}{\partial x^2}(x,y)+\frac{\partial^2 u}{\partial y^2}(x,y)\right]+\b[u(x,-y)-u(x,y)]=0,\]
it can be rewritten in a convenient way as
\[Lu:=\a\Delta u+\b( A^*-\Id) u=0,\]
where
\[\Delta=\frac{\partial^2 }{\partial x^2}+\frac{\partial^2 }{\partial y^2}\text{ and } A=\begin{pmatrix}1 & 0  \\ 0 & -1\end{pmatrix}.\]
If we think of a circular plate in which the boundary is constantly cooled and the surface has a constant heat source given by a function $h$, we are imposing Dirichlet boundary conditions in the ball $B$ of radius $\rho\in \bR^+$ and considering the problem
\begin{equation}\label{eqrefdisk}Lu=h,\ u|_{\partial B}=0.\end{equation}
Observe that, $\Delta$, expressed in tensor notation, is $\Delta=D_{\omega^0_2}$ where
\[\omega^0_2=\frac{1}{2}\left[(1,1)\odot(1,1)+(1,-1)\odot(1,-1)\right].\]
Besides, $A\omega^0_2=\omega^0_2$ and, thus, $\Delta A^*=A^*\Delta$. Hence, using Theorem \ref{theoremO2}, we have to take $R=-\a\Delta+\b A^*+\b\Id $ and thus
\begin{align*}RL= & -\a^2\Delta^2-\a\b A^*\Delta+\a\beta\Delta+\a\beta A^*\Delta+\b^2\Id-\b^2A^*+\a\b\Delta+\beta^2A^*-\b^2\Id \\ = & -\a^2\Delta^2+2\a\b\Delta=(-\a^2\Delta+2\a\b\Id)\Delta.\end{align*}

Now, the boundary conditions transformed by $R$ are
\[0=Ru=-\a\Delta u+\b A^*u+\b u=-\a\Delta u, \]
that is, the reduced problem becomes
\begin{equation}\label{redeq}
RLu=Rh=:f,\ u|_{\partial B}=0,\ \Delta u|_{\partial B}=0,
\end{equation}
which is equivalent to the sequence of problems
\begin{align}
\Delta u= & v,\ u|_{\partial B}=0,\label{1stprob}\\
(-\a^2\Delta+2\a\b\Id)v= & f, \ v|_{\partial B}=0.\label{2ndprob}
\end{align}

Problem \eqref{1stprob} is the well-known Poisson equation with Dirichlet conditions on the circle of radius $\rho$. The Green's function can be written in polar coordinates as
\begin{displaymath}
G_1(r,\varphi,\tilde r,\tilde \varphi)=\frac{-1}{4\pi}\ln\left[\frac{r^2\tilde r^2 -2\rho^2 r\tilde r\cos(\varphi-\tilde\varphi)+\rho^4}{ \rho^2r^2-2 \rho^2 r\tilde r\cos(\varphi-\tilde\varphi)+ \rho^2\tilde r^2}\right].
\end{displaymath}
See \cite[Section 7.2.3]{Poly}. On the other hand, problem \eqref{2ndprob} is a Helmholtz equation, and the Green's function can be described in terms of the eigenfunctions of the associated homogeneous problem (see \cite[Section 7.3.3]{Poly}). More concretely, the associated Green's function in polar coordinates is written as
\begin{align*}
& G_2(r,\varphi,\tilde r,\tilde \varphi) \\ = & \frac{1}{\alpha^2}\sum_{n=0}^{\infty}\sum_{m=1}^{\infty}\frac{1}{\left(\frac{\mu_{nm}^2}{\rho^2}+\frac{2\beta}{\alpha}\right)\|w_{nm}^{(1)}\|^2}
\left[w_{nm}^{(1)}(r,\varphi)w_{nm}^{(1)}(\tilde r,\tilde \varphi)+w_{nm}^{(2)}(r,\varphi)w_{nm}^{(2)}(\tilde r,\tilde \varphi) \right],
\end{align*}
where $\mu_{nm}$ are the positive zeroes of the Bessel functions $J_n$, the eigenfunctions are given by
\begin{displaymath}
w_{nm}^{(1)}=J_n\left(\frac{\mu_{nm}}{\rho}r\right)\cos n\varphi,\quad w_{nm}^{(2)}=J_n\left(\frac{\mu_{nm}}{\rho}r\right)\sin n\varphi,
\end{displaymath}
and 
\begin{displaymath}
\|w_{nm}^{(1)}\|^2=\frac{1}{2}\pi \rho^2 (1+\delta_{n\,0})\left[J_{n}'(\mu_{nm})\right]^2,
\end{displaymath}
where $\delta_{ij}=1$ if $i=j$ and $0$ if $i\ne j$.

Now, the Green's function associated to problem \eqref{redeq} is given by
\begin{displaymath}
G_3(r,\varphi,\tilde r,\tilde \varphi)=\int_{0}^{\rho}\int_{0}^{2\pi}G_2(r,\varphi,\hat r,\hat \varphi)G_1(\hat r,\hat\varphi,\tilde r,\tilde \varphi) \dif\hat\varphi  \dif \hat r. 
\end{displaymath}
In conclusion, the Green's function related to problem \eqref{eqrefdisk} is 
\[G_4(\eta,\xi)=R_\vdash G_3(\eta,\xi)=\int_{0}^{\rho}\int_{0}^{2\pi}R_\vdash G_2(r,\varphi,\hat r,\hat \varphi)G_1(\hat r,\hat\varphi,\tilde r,\tilde \varphi) \dif\hat\varphi  \dif \hat r,\]
where $R_\vdash$ has to be expressed in polar coordinates in order to act in the first two variables of $G_3$:
\[R=-\a\left[\frac{\partial^2}{\partial r^2}+\frac{1}{r}\frac{\partial}{\partial r}+\frac{1}{r^2}\frac{\partial^2}{\partial \varphi^2}\right]+\b A^*+\b\Id.\]
Also, it is known that $J_n'(z)=(n/z)J_n(z)-J_{n+1}(z)$, so
\begin{align*} &  R_\vdash G_2(r,\varphi,\hat r,\hat \varphi) \\ = & \frac{1}{\alpha^2}\sum_{n=0}^{\infty}\sum_{m=1}^{\infty}\frac{1}{\left(\frac{\mu_{nm}^2}{\rho^2}+\frac{2\beta}{\alpha}\right)\|w_{nm}^{(1)}\|^2}
\left[\til w_{nm}^{(1)}(r,\varphi)w_{nm}^{(1)}(\tilde r,\tilde \varphi)+\til w_{nm}^{(2)}(r,\varphi)w_{nm}^{(2)}(\tilde r,\tilde \varphi) \right],\end{align*}
where
\begin{align*}
\til w_{nm}^{(1)}  = & \left.\left(\left(\frac{\mu_{nm}}{\rho}\right)^2\left[\(\frac{n\rho}{\mu_{nm}r}\right)^2J_n-\left(1+\frac{(n+1)\rho}{\mu_{nm}r}\right)J_{n+1}+J_{n+2}\right]\right.\right.\\ & \left.\left.+\frac{n}{r}\left[\frac{\rho}{\mu_{nm}r} J_n-J_{n+1}\right] -n^2\left(\frac{\mu_{nm}}{\rho}r\right)^{-2}J_n\right)\right|_{\left(\frac{\mu_{nm}}{\rho}r\right)} \cos n\varphi, \\ \til w_{nm}^{(2)}  = & \left.\left(\left(\frac{\mu_{nm}}{\rho}\right)^2\left[\(\frac{n\rho}{\mu_{nm}r}\right)^2J_n-\left(1+\frac{(n+1)\rho}{\mu_{nm}r}\right)J_{n+1}+J_{n+2}\right]\right.\right.\\ & \left.\left.+\frac{n}{r}\left[\frac{\rho}{\mu_{nm}r} J_n-J_{n+1}\right] -n^2\left(\frac{\mu_{nm}}{\rho}r\right)^{-2}J_n\right)\right|_{\left(\frac{\mu_{nm}}{\rho}r\right)} \sin n\varphi.
\end{align*}
\begin{exa} Inspired by the previous problem, we now change the term due to Newton's law of cooling by a diffusion term in the following way.
	\[\frac{\partial{K}}{\partial t}(t,x,y)=\a\left[\frac{\partial^2 K}{\partial x^2}(t,x,y)+\frac{\partial^2 K}{\partial y^2}(t,x,y)\right]+\b\left[\frac{\partial^2 K}{\partial x^2}(t,x,-y)+\frac{\partial^2 K}{\partial y^2}(t,x,-y)\right],\]
	where $\a,\b>0$, $\b\ne\a$.\par
	If we consider the associated stationary problem
	\[\a\left[\frac{\partial^2 K}{\partial x^2}(x,y)+\frac{\partial^2 K}{\partial y^2}(x,y)\right]+\b\left[\frac{\partial^2 K}{\partial x^2}(x,-y)+\frac{\partial^2 K}{\partial y^2}(x,-y)\right]=0,\]
	it can be rewritten as
	\[LK:=\a\Delta K+\b A^*\Delta K=0,\]
Using Theorem \ref{theoremO2}, we take $R=-\a\Delta +\b A^*\Delta $ and then
	\[RL=-\a^2\Delta^2-\a\b\Delta A^*\Delta+\b\a A^*\Delta^2+\b^2(A^*\Delta)^2=\b^2\Delta^2-\a^2\Delta^2=(\b^2-\a^2)\Delta^2.\]
	Now, if we consider the fundamental solution of the bi-Laplacian $\Delta^2$ \cite[equation~(2.61)]{Gaz} we obtain a Green's function given by
	\[G_1(\eta,\xi)=\frac{1}{8\pi}\|\eta-\xi\|^2\ln\|\eta-\xi\|,\  \eta,\xi\in\bR^2.\]
	Hence, in that case, the Green's function associated to $L$ is given by
	\[G_2(\eta,\xi)=R_\vdash G_1(\eta,\xi)=(\b-\a)\frac{\ln \|\eta-\xi\|+1}{2 \pi },\ \eta,\xi\in\bR^2.\]
	If we consider the problem
	\[LK=h,\ u|_{\partial B}=0,\]
	the reduced problem becomes
	\begin{equation}
	\label{redprobc}
	(\b^2-\a^2)\Delta^2K=h,\ u|_{\partial B}=0,\ Ru|_{\partial B}=0.\end{equation}
	Now, the condition $Ru=-\a\Delta u +\b A^*\Delta u=0$ is satisfied if we can guarantee that $\Delta u=0$, so we can consider the problem
	\begin{equation}\label{redprobc2}(\b^2-\a^2)\Delta^2K=h,\ u|_{\partial B}=0,\ \Delta u|_{\partial B}=0.
	\end{equation}
	For problem \eqref{redprobc2} we have that the Green's function is given by
	\[G_3(\eta,\xi)=\frac{1}{8\pi}\|\eta-\xi\|^2\(\ln \rho-1+\ln\|\eta-\xi\|\)+\frac{\rho^2}{8\pi},\  \eta,\xi\in\bR^2.\]
	Hence, the Green's function related to problem \eqref{redprobc} is 
	\[G_4(\eta,\xi)=\frac{ \ln \rho+\ln \|\eta-\xi\|}{2 \pi }.\]
	In general, the functions
	\[G_5(\eta,\xi)=\frac{1}{8\pi}\|\eta-\xi\|^2\( \mu+\ln\|\eta-\xi\|\)+\frac{\nu}{8\pi},\  \eta,\xi\in\bR^2,\]
	with $\mu,\,\nu\in\bR$, are Green's functions related to the operator $\Delta^2$ with different boundary conditions. The associated function for the operator $L$ is given by
	\[G_6(\eta,\xi)=\frac{1+ \mu+\log \|\eta-\xi\|}{2 \pi }.\]
	\end{exa}


\section*{Acknowledgements}
	The authors would like to recognize their gratitude towards the developers of the \textit{Mathematica} \textbf{NCAlgebra} software \cite{NCAlgebra}, which allowed us to check the results of this paper computationally and to Professor Santiago Codesido for providing useful insight about the model presented in this paper. Also, Adri\'an Tojo would like to acknowledge his gratitude towards the Department of Applied Mathematics of the University of Granada where this work was written, and specially towards the coauthor
	of this paper, for his affectionate reception, this and other times, and his invaluable work and advice.

\end{document}